\documentclass[11pt]{amsart}
\usepackage{amsmath}
\usepackage{a4wide}
\usepackage[utf8]{inputenc}
\usepackage{amssymb}
\usepackage{amsopn}
\usepackage{epsfig}
\usepackage{amsfonts}
\usepackage{latexsym}
\usepackage{amsthm}
\usepackage{enumerate}
\usepackage[UKenglish]{babel}
\usepackage{verbatim}
\usepackage{color}          
\usepackage{mathrsfs}
\usepackage{pgf}
\usepackage{tikz}

\usepackage{mathrsfs}   

\usetikzlibrary{arrows,automata}

\DeclareMathAlphabet{\pazocal}{OMS}{zplm}{m}{n}
\newtheorem{theorem}{Theorem}[section]
\newtheorem{lemma}[theorem]{Lemma}
\newtheorem{proposition}[theorem]{Proposition}

\theoremstyle{definition}

\theoremstyle{remark}
\newtheorem{remark}[theorem]{Remark}

\numberwithin{equation}{section}

\newcommand{\R}{\ensuremath{\mathbb{R}}}
\newcommand{\N}{\ensuremath{\mathbb{N}}}

\newcommand{\set}[1]{\left\{#1\right\}}

\newcommand{\ga}{\gamma}

\newcommand{\f}{\infty}

\newcommand{\Om}{\Omega}

\begin{document}
\title[Self-similar sets with complete overlaps]{On a kind of self-similar sets with complete overlaps}

\author[D. Kong]{Derong~Kong}
\address[D. Kong]{College of Mathematics and Statistics, Chongqing University, 401331, Chongqing, P.R.China}
\email{derongkong@126.com}

\author[Y. Yao]{Yuanyuan~Yao$^*$}
\address[Y. Yao]{Department of Mathematics, East China University of Science and Technology, Shanghai 200237, P.R. China}
\email{yaoyuanyuan@ecust.edu.cn}

 \subjclass[2010]{{Primary: 28A80, Secondary: 11A63, 28A78}}
\begin{abstract}
Let $E$ be the self-similar set generated by the {\it iterated function system}
{\[
f_0(x)=\frac{x}{\beta},\quad f_1(x)=\frac{x+1}{\beta}, \quad f_{\beta+1}=\frac{x+\beta+1}{\beta}
\]}with $\beta\ge 3$. {Then} $E$ is a self-similar set with   complete {overlaps}, i.e., $f_{0}\circ f_{\beta+1}=f_{1}\circ f_1$, but $E$ is not totally self-similar.
 We investigate all its generating iterated function systems, give the spectrum of $E$, and determine the Hausdorff dimension and Hausdorff measure of $E$ and of the sets  which contain all points in $E$ having finite or infinite different triadic codings.
 \end{abstract}
 \keywords{Iterated function system; self-similar set; complete overlap; spectrum; multiple codings.}
 \thanks{$^*$ Corresponding author.}
\maketitle

\section{introduction}\label{sec:introduction}

Let $\beta\ge 3$, and let $E_\beta$ be the self-similar set generated by the \emph{iterated function system} (IFS)
\[f_d(x)=\frac {x+d}{\beta},\quad d\in\set{0,1,\beta+1}.\]
Then $E_\beta$ is the unique non-empty compact set in the real line satisfying $E_\beta=\cup_{d\in {\set{0,1,\beta+1}}}f_d(E_\beta)$ (cf. \cite{Hutchinson}).
It is easy to check   $f_{0}\circ f_{\beta+1}=f_{1}\circ f_1$. Then the self-similar set $E_\beta$ has   complete {overlaps}.

Our interest in $E_\beta$ comes from expansions in non-integer bases (for the surveys see \cite{Komornik,Sidorov}). One example is expansions with digit set $\{0,1,\beta\}$. For $\beta>1$, let $F_\beta$ be the attractor of the IFS
{\[ \phi_d(x)=\frac{x+d}{\beta},\quad d\in\{0,1,\beta\}.\]}Then $F_\beta$ is a self-similar set with overlaps since $\phi_0(F_\beta)\cap \phi_1(F_\beta)\ne \emptyset$.
There has been considerable interest in $F_\beta$. For example, Ngai and Wang \cite{NW} investigated the Hausdorff dimension of $F_\beta$. Zou et al.~\cite{ZLL} considered the set of points in $F_\beta$ having a unique $\beta$-expansion. Yao and Li \cite{YL} gave all the generating IFSs of $F_\beta$. Dajani et al.~\cite{DJKL} described the size of the set of bases $\beta$ for which there exists $x\in F_\beta$ having finite or countably many different $\beta$-expansions and the set of $x\in F_\beta$ which have exactly finite or countable $\beta$-expansions.

There are two striking differences between $E_\beta$ and $F_\beta$, one is that the {\it  {total self-similarity}} (see \cite{BMS} for its first appearance) fails in $E_\beta$,
as we will explain later. Another is that by the obvious fact that $\phi_0\circ \phi_\beta=\phi_1\circ \phi_0$ we have $\phi_{1^k0}=\phi_{0\beta^k}$ for any positive integer $k$, {which is an important property   in discussing $F_\beta$}. However, we do not see this property in $E_\beta$.

We will abbreviate $E_\beta$ to $E$ {if no confusion arises}. Observe that for each point $x\in E$ there exists an infinite sequence $(d_i)\in\set{0,1,\beta+1}^\N$ such that
 \begin{equation}\label{eq:coding-map}
 x=\lim_{n\to\f}f_{d_1\dots d_n}(0):=\lim_{n\to\f}f_{d_1}\circ\cdots\circ f_{d_n}(0)=\sum_{i=1}^\f\frac{d_i}{\beta^i}.
 \end{equation}
 The infinite sequence $(d_i)$ is called a \emph{coding} of $x$. Since $f_{0(\beta+1)}=f_{11}$, a point $x\in E$ may have multiple codings.
By (\ref{eq:coding-map}) it follows that
\[
E=\set{\sum_{i=1}^\f\frac{d_i}{\beta^i}: d_i\in\Omega},
\]
where $\Omega:=\set{0,1,\beta+1}$ is the \emph{alphabet} which will be fixed throughout the paper.

 \begin{figure}[h!]
\begin{center}
\begin{tikzpicture}[
    scale=9,
    axis/.style={very thick},
    important line/.style={thick},
    dashed line/.style={dashed, thin},
    pile/.style={thick, ->, >=stealth', shorten <=2pt, shorten
    >=2pt},
    every node/.style={color=black}
    ]
    \draw[axis] (0,0)  -- (1.5,0) node(xline)[right]{};

      \node[] at (0, 0.05){$0$};

      \node[] at (1.6, 0){$\Delta$};
      \node[] at (1.48, 0.05){$\ga_\beta:=\frac{\beta+1}{\beta-1}$};

      \draw[axis] (0,-0.2)  -- (0.3,-0.2) node(xline)[right]{};
 \node[] at (0.15, -0.15){$f_0$};
     \draw[axis] (0.2,-0.21)  -- (0.5,-0.21) node(xline)[right]{};
 \node[] at (0.35, -0.16){$f_1$};
    \draw[axis] (1.2,-0.2)  -- (1.5,-0.2) node(xline)[right]{};
 \node[] at (1.35, -0.15){$f_{\beta+1}$};
  \node[] at (1.6, -0.2){$\Delta_1$};

    \draw[axis] (0,-0.4)  -- (0.06,-0.4) node(xline)[right]{};
 \node[] at (0.03, -0.35){$f_{00}$};
     \draw[axis] (0.04,-0.41)  -- (0.1,-0.41) node(xline)[right]{};
 \node[] at (0.07, -0.44){$f_{01}$};
    \draw[blue,axis] (0.24,-0.4)  -- (0.3,-0.4) node(xline)[right]{};
 \node[blue] at (0.27, -0.35){$f_{0(\beta+1)}=f_{11}$};

 \draw[axis] (0.2,-0.41)  -- (0.26,-0.41) node(xline)[right]{};
 \node[] at (0.23, -0.44){$f_{10}$};
  \draw[axis] (0.44,-0.41)  -- (0.5,-0.41) node(xline)[right]{};
 \node[] at (0.47, -0.44){$f_{1(\beta+1)}$};

   \draw[axis] (1.2,-0.4)  -- (1.26,-0.4) node(xline)[right]{};
 \node[] at (1.23, -0.35){$f_{(\beta+1)0}$};
     \draw[axis] (1.24,-0.41)  -- (1.3,-0.41) node(xline)[right]{};
 \node[] at (1.27, -0.44){$f_{(\beta+1)1}$};
    \draw[axis] (1.44,-0.4)  -- (1.5,-0.4) node(xline)[right]{};
 \node[] at (1.47, -0.35){$f_{(\beta+1)(\beta+1)}$};

  \node[] at (1.6, -0.4){$\Delta_2$};

\end{tikzpicture}
\end{center}
\caption{The first three levels $\Delta, \Delta_1, \Delta_2$ of basic intervals of $E_\beta$ with $\beta=5$.}\label{fig:1}
\end{figure}
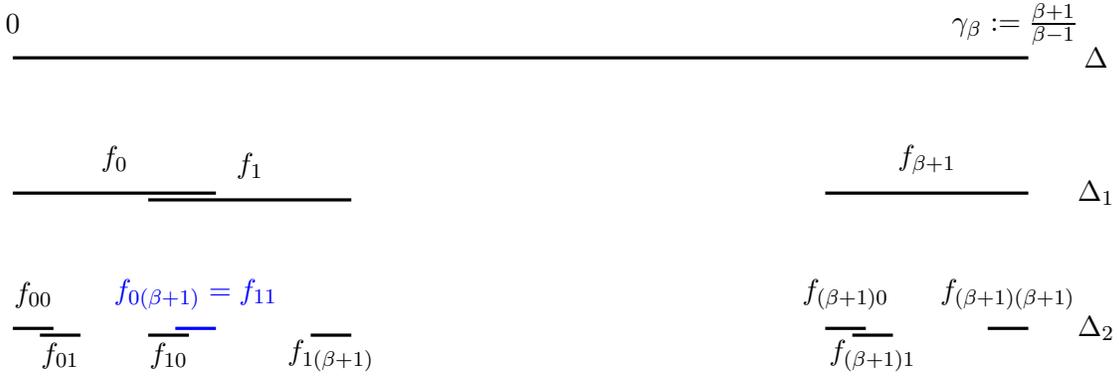


Let $\Delta=[0,\ga_\beta]$ be the convex hull of $E$,  where $\ga_\beta:=(\beta+1)/(\beta-1)$. Then  for $n\in\N$ let
\[\Delta_n:=\bigcup_{{\bf i}\in\Omega^n}f_{\bf i}(\Delta)\] be the union of all level-$n$ basic intervals of $E$  (see Figure \ref{fig:1} for the  first three levels of basic intervals), where $\Omega^n:=\set{d_1\ldots d_n: d_i\in\Omega}$ is the set of all blocks of length $n$  with respect to {the} alphabet $\Omega$.  Since $\beta\ge 3$, the basic interval $f_{10}(\Delta)$ is on the righthand side of  $f_{01}(\Delta)$, and their intersection is nonempty {(in fact is a singleton only if $\beta=3$)}.

We write
\[H:=\Delta\backslash \Delta_1=(f_1(\ga_\beta), f_{\beta+1}(0)),\quad \textrm{and}\quad H_{\bf i}=f_{\bf i}(H).\]
 Then $H$ and $H_{\bf i}$ are open intervals for all ${\bf i}\in\Omega^*:=\bigcup_{n=1}^\f \Omega^n$.
By a {\it hole} of $E$ we mean an open interval $(a,b)$ with $a,b\in E$ and $(a,b)\cap E=\emptyset$.
Then $H$ is the biggest hole of $E$.
Unlike the situation in $F_\beta$, we emphasize that $H_{\bf i}$ is not necessarily a hole of $E$. For example, $H_0=f_0(H)$ is not a hole since $H_0\cap f_{10}(E)\ne\emptyset$. By \cite[Proposition 2.1]{DKY}  it follows that $E$ is not {\it totally self-similar}.

As a new class of self-similar sets with complete overlaps we first investigate all the generating iterated function systems (IFSs) of $E$.
We call a nonempty compact set
$F\subseteq {\mathbb R}^d$ a {\it self-similar set} if it is a finite union of its self-similar copies; this is, there exists a family of contractive similitudes
${\mathcal F}=\{\psi_i(x)=\rho_iO_ix+b_i\}_{i=1}^N$ ($N\ge 2$ is an integer) such that $F=\bigcup_{i=1}^N\psi_i(F)$, where $\rho_i\in(0,1),O_i$ is a {$d\times d$}
orthogonal matrix and $b_i{\in \mathbb R^d}$ is a translation. The family ${\mathcal F}$ is called a generating IFS for $F$. It is well known that ${\mathcal F}$ determines
$F$ uniquely, but not vice versa.

The question of determining all the generating IFSs for a given self-similar set was initiated by Feng and Wang \cite{FW1} and was later studied by Deng and Lau \cite{DL1,DL2}, and Yao \cite{Yao}. The requirement
of some separation condition (open set condition or strong separation {condition}, etc.) is crucial in their proof. If the assumption of separation condition is dropped, the problem will be more complicated. Dajani et al.~\cite{DKY} first {gave} an answer to this problem for a kind of self-similar sets which are totally self-similar.


Our first  main result is on the  generating IFSs of $E$, which satisfies neither of the conditions stated in the above paragraph.
\begin{theorem}\label{ifs}
If $g$ is an affine map such that $g(E)\subseteq E$, then $g=f_{\bf i}$ for some ${\bf i}\in\Omega^*$.
\end{theorem}

During the past fifty years the question on the spectrum {of overlapping fractals} have been extensively studied because of its close connections with infinite Bernoulli convolutions and
expansions in noninteger bases.  Many works have been devoted to this topic, see \cite{AK,EK,Feng,FW2,PS} {and  the references therein}.

{Motivated by the work of \cite{DKY} we define the \emph{spectrum} of $E$ by
\begin{eqnarray*}
\ell_\beta:=\inf\left\{\left|\sum_{i=0}^{n-1}d_i\beta^i\right|\ne 0:d_i\in\Om-\Om=\{0,\pm 1,\pm \beta,\pm (\beta+1)\},n\in{\mathbb N}\right\}.
\end{eqnarray*}
In our second result we show that   the spectrum $\ell_\beta$ of $E$   is constant for all $\beta\ge 3$.}
\begin{proposition}\label{prop:spectrum}
For any $\beta\ge 3$, we have $\ell_\beta=1$.
\end{proposition}

Note that $E$ is a self-similar set having overlaps. Then a point $x\in E$ may have multiple codings. Motivated by the work of \cite{DJKL} we introduce the following subsets of $E$. For $k\in\N\cup\set{\aleph_0,  2^{\aleph_0}}$ let
\[
E^{(k)}:=\set{x\in E: x\textrm{ has precisely }k\textrm{ different codings with alphabet }\Om}.
\]
Then $E=E^{(2^\aleph_0)}\cup E^{(\aleph_0)}\cup\bigcup_{k=1}^\f E^{(k)}$.
Our final result is on the Hausdorff dimension  and Hausdorff measures of $E$ and $E^{(k)}$.

\begin{proposition}\label{prop:dim-E-U}\begin{itemize}
\item[{\rm(i).}] The Hausdorff dimensions of $E$ and $E^{(2^{\aleph_0})}$ are given by
\[
\dim_H E=\dim_H E^{(2^\aleph_0)}=\frac{\log(3+\sqrt{5})-\log 2}{\log \beta}=:s.
\]
Furthermore, $\mathcal H^s(E)=\mathcal H^s(E^{(2^\aleph_0)})\in(0, \f)$.

\item[{\rm(ii)}] If $k=\aleph_0$ or $k\in\N$ is not of the form $2^m$, then $E^{(k)}=\emptyset$. Otherwise, the set $E^{(2^m)}$ with $m\in\set{0,1,2,\ldots}$ has the same Hausdorff dimension given by
\[
\dim_H E^{(2^m)}=\frac{\log r}{\log \beta}=:t,
\]
where $r\approx 2.24698$ is the root of $x^3-2x^2-x+1=0$. Furthermore, $\mathcal H^t(E^{(1)})\in(0,\f)$, and $\mathcal H^t(E^{(2^m)})=\f$ for any $m\in\N$.

\end{itemize}
\end{proposition}

The rest of the paper is arranged as follows. In the next section we discuss all the generating IFSs of $E$ and establish Theorem  \ref{ifs}. In the last Section we investigate the spectrum, the Hausdorff dimension and Hausdorff measure of $E$ and of the sets  of points in $E$ having finite or infinite different $\beta$-expansions, and prove both Proposition \ref{prop:spectrum} and Proposition
\ref{prop:dim-E-U}.

\section{Generating iterated function systems of $E$}\label{sec:generating-IFS}

 In this section we will investigate all generating iterated function systems of $E$, and prove Theorem \ref{ifs}.
First we prove  the asymmetry of $E$.
\begin{lemma}\label{notin}
$-E+c \nsubseteq E$ for any $c\in\R$.
\end{lemma}
\begin{proof}Suppose on the contrary that $-E+c\subseteq E$ for some $c\in\R$. Recall that $0$ is the minimum of $E$, and $\ga_\beta$ is the maximum of $E$. Then \begin{equation}\label{eq:kong-1}
-E+c\subseteq \left[c-\ga_\beta,c\right]\quad \textrm{with}\quad c-\ga_\beta\textrm{ and }c\in E.
\end{equation}
If $c\ne \ga_\beta$, then
one of $c$ and $\ga_\beta-c$ is not in $[0, \ga_\beta]$, therefore is not in $E$, leading to a contradiction with (\ref{eq:kong-1}).
So we must have  $c=\ga_\beta$.
Then
\begin{equation}\label{eq:kong-2}-E+\ga_\beta\subseteq E.
\end{equation}

However, by using $\beta\ge 3$ it is easy to check that
\[-f_{1(\beta+1)}(0)+\ga_\beta=-\frac{2\beta+1}{\beta^2}+\frac{\beta+1}{\beta-1}\in \left(\frac{2}{\beta-1},\frac{\beta+1}{\beta}\right)=H,\] where $H$ is the largest  hole of $E$. This leads to a contradiction with (\ref{eq:kong-2}). Therefore, $-E+c\nsubseteq E$ for any $c\in\R$.
\end{proof}

In the next lemma we show that for an affine map $g(x)=\mu x+b$, if $g(E)=\mu E+b\subseteq E$, then $g(E)$ can not intersect both $f_1(E)$ and $f_{\beta+1}(E)$.


\begin{lemma}\label{sep}
If $g(E)\subseteq E$, then either $g(E)\subseteq f_0(E)\cup f_1(E)$ or $g(E)\subseteq f_{\beta+1}(E)$.
\end{lemma}
\begin{proof}It suffices to prove that $g(E)$ can not intersect both $f_0(E)\cup f_1(E)$ and $f_{\beta+1}(E)$.
Conversely, assume that
\begin{equation*}
g(E)\cap (f_0(E)\cup f_1(E))\ne\emptyset\quad{\textrm{and}}\quad g(E)\cap f_{\beta+1}(E)\ne \emptyset.
\end{equation*}
Then both $E$ and $\Delta=[0, \ga_\beta]$ are the attractors of the IFS $\{f_0,f_1,f_{\beta+1},g\}$. So they should be the same, leading to a contradiction.
\end{proof}

The following lemma  plays an important role in the proof of Theorem \ref{ifs}. For a compact set $A\subseteq \R$ denote by $A_{\max}$ and $A_{\min}$ the largest and smallest elements of $A$, respectively. Then $E_{\min}=0$ and $E_{\max}=\ga_\beta$.
\begin{lemma}\label{belong}
Assume that $\mu E+b\subseteq E$.
\begin{itemize}
\item[{\rm(i)}] If $(\mu E+b)_{\max}<f_{10}(0)$, then $\mu E+b\subseteq f_0(E)$;

\item[{\rm(ii)}] If
$f_{01}(\ga_\beta)<(\mu E+b)_{\min}< f_{\beta+1}(0)$, then $\mu E+b\subseteq f_1(E)$;

\item[{\rm(iii)}] If $(\mu E+b)_{\min}> f_{1}(\ga_\beta)$, then $\mu E+b\subseteq f_{\beta+1}(E)$.
\end{itemize}
\end{lemma}
\begin{proof}
In view of Figure \ref{fig:1}, (i) and (iii) are obvious. We only need to prove (ii).

Suppose $\mu E+b\subseteq E$ and $f_{01}(\ga_\beta)<(\mu E+b)_{\min}< f_{\beta+1}(0)$. By Lemma  \ref{sep} and $f_{0(\beta+1)}=f_{11}$ it follows that
\begin{align*}
\mu E+b&\subseteq (f_0(E)\cup f_1(E))\cap f_1(\Delta)\\&=(f_{0(\beta+1)}(E)\cap f_1(\Delta))\cup f_1(E)=(f_{11}(E)\cap f_1(\Delta))\cup f_{1}(E)\subseteq f_1(E).
\end{align*}
This completes our proof.
\end{proof}

\subsection{Determination of $\mu$ for $\mu E+b\subseteq E$}
Assume $\mu E+b\subseteq E$. Since $E$ is a compact set, it follows that $|\mu|\le 1$. In view of Lemma \ref{notin} it is only necessary to consider
$0<|\mu|<1$ and $b\in\R$. We will show that the only possibility is $\mu=\beta^{-n}$ for some $n\in\N$.
\begin{proposition}\label{prop:contration-ratio}
If $g(E)=\mu E+b\subseteq E$ with $0<|\mu|<1$ and $b\in\R$, then $\mu=\beta^{-n}$ for some  $n\in\N$.
\end{proposition}

The proof of Proposition \ref{prop:contration-ratio}  will be split into several lemmas.
\begin{lemma}\label{mu1}
Let $0<|\mu|<1$. If $\mu E+b\subseteq E$, then  $0<|\mu|<1/\beta$ or $\mu=1/\beta$.
\end{lemma}
\begin{proof}
Suppose $\mu E+b\subseteq E$.  Then by Lemma \ref{sep} we have
\[\textrm{either}\quad \mu E+b\subseteq \frac{E+\beta+1}{\beta}\quad\textrm{or}\quad  \mu E+b\subseteq \frac{E}\beta\cup
\frac{E+1}{\beta}.\]
 In the first case, by observing the lengths of the convex hulls on both sides of the inclusion we have $0<|\mu|\le  1/\beta$. In the following we will discuss the latter case
 \begin{equation}\label{eq:kong-5}
 \mu E+b\subseteq \frac{E}\beta\cup
\frac{E+1}{\beta}=f_0(E)\cup f_{1}(E).
\end{equation}

Note that each hole of $\mu E+b$ is mapped by a hole of $E$, and the largest hole $H$ of $E$ has length $A:=f_{\beta+1}(0)-f_1(\ga_\beta)$.

(I). We first consider the case that  $\mu E+b$ contains a hole including the interval $H_1=f_1(H)$. Say  $(c,d)$ is the gap of $E$ such that
\begin{equation}\label{eq:kong-3}
\mu(c, d)+b\supseteq H_1.
\end{equation}
 Then we claim that
 \begin{equation}\label{eq:kong-4}
 (c,d)=H.
 \end{equation} If $(c, d)\ne H$, then by noting that  the second largest hole of $E$ has length $ A/\beta$, it follows that
 \[
 |\mu|(d-c)\le \frac{|\mu|A}{\beta}<\frac{A}{\beta},
 \]
 leading to contradiction with (\ref{eq:kong-3}). This proves (\ref{eq:kong-4}).
\begin{itemize}\item
If $\mu>0$, then by (\ref{eq:kong-5}) and (\ref{eq:kong-4}) it follows that
$\mu f_{\beta+1}(E)+b\subseteq f_{1(\beta+1)}(E).$ This implies  $\mu \in(0,   1/\beta]$.

\item If $\mu<0$, then again by (\ref{eq:kong-5}) and (\ref{eq:kong-4})  we have $\mu f_0(E)+b\subseteq f_{1(\beta+1)}(E)$. This gives $\mu\in[-1/\beta,0)$. Suppose $\mu=- 1/\beta$. Then
\[-f_{00}(E)+b=-\frac{1}{\beta}f_0(E)+b\subseteq f_{1(\beta+1)}(E).\]
This will lead to a contradiction with Lemma  \ref{notin}. So, $\mu\in(-1/\beta, 0)$.
\end{itemize}

(II). Next we consider the case that $\mu E+b$ does not contain a hole including  $H_1.$  Then by (\ref{eq:kong-5}) we either have $\mu E+b\subseteq f_{1(\beta+1)}(E)$ which implies $0<|\mu|\le {\beta^{-2}}<1/\beta$,  or
\begin{equation}\label{eq:kong-6}
\mu E+b\subseteq f_{00}(E)\cup f_{01}(E)\cup
f_{10}(E)\cup f_{11}(E)= f_0(\Delta)\cap E
\end{equation}
which also gives $0<|\mu|\le 1/{\beta}$. In the following it suffices to prove $\mu\ne -1/\beta$.

Suppose on the contrary that $\mu=-1/\beta$. Then by (\ref{eq:kong-6}) we have
$- \beta^{-1}E+b\subseteq f_0(\Delta)\cap E.$
By the same argument as in the proof of Lemma \ref{notin} it follows that  $b=\ga_\beta/\beta$.
Therefore,
\begin{equation}\label{eq:kong-7}
-\frac E{\beta^2}+\frac{\ga_\beta}{\beta}\subseteq -\frac E{\beta}+\frac{\ga_\beta}{\beta}\subseteq E.
\end{equation}
Note by $\beta\ge 3$ that
\[\left(-\frac E{\beta^2}+\frac{\ga_\beta}{\beta}\right)_{\min}=\frac {\ga_\beta(\beta-1)}{\beta^2}=\frac{\beta+1}{\beta^2}
>\frac 2{\beta(\beta-1)}=f_{01}(\ga_\beta).\]
 Then by (\ref{eq:kong-6}), (\ref{eq:kong-7}) and Lemma  \ref{belong} (ii) it follows that
$-\beta^{-2} E +\beta^{-1} \ga_\beta \subseteq f_1(E)=\beta^{-1}({E+1}),$
 which is equivalent to
\[-\frac E\beta+\frac 2{\beta-1}\subseteq E.\]

So,
 \[-\frac E{\beta^2}+\frac 2{\beta-1}\subseteq -\frac E{\beta}+\frac 2{\beta-1}\subseteq E.\]
Note that
\begin{equation*}
\left(-\frac E{\beta^2}+\frac 2{\beta-1}\right)_{\min}=-\frac {\ga_\beta}{\beta^2}+\frac 2{\beta-1}\in\left(f_{01}(\ga_\beta), f_{\beta+1}(0)\right).
\end{equation*}
By Lemma  \ref{belong} (ii) it follows that
\[-\frac E{\beta^2}+\frac 2{\beta-1}\subseteq f_1(E)=\frac{E+1}\beta.\]  This implies
$-\beta^{-1} E+\ga_\beta\subseteq E.$
One can check that $(-\beta^{-1}E+\ga_\beta)_{\min}>f_1(\ga_\beta)$. By Lemma  \ref{belong} (iii) we can deduce
\[-\frac E\beta+\ga_\beta\subseteq f_{\beta+1}(E)=\frac{E+\beta+1}{\beta}.\]
 This contradicts to Lemma  \ref{notin}. So $\mu\ne-1/\beta$, and we complete the proof.
\end{proof}

In the following we will show that if $\mu E+b\subseteq E$ with $0<|\mu|<1/\beta$, then $\beta\mu E+c\subseteq E$ for some $c\in\R$ (see Lemma \ref{mu2}). To prove this we need the following two lemmas.

\begin{lemma}\label{mul0}
Let $\mu E+b\subseteq E$ with $0<\mu< 1/\beta$. If
\begin{equation}\label{eq:kong-11}
 b\ge\frac{1}{\beta}-\frac{2\mu}{\beta-1}\quad\textrm{and}\quad b>\frac{1-\mu-\mu\beta}{\beta-1},
 \end{equation}
 then
$\beta\mu E+c\subseteq E$ for some $c\in {\mathbb R}$.
\end{lemma}
\begin{proof}
Suppose $\mu E+b\subseteq E$. By Lemma \ref{sep} it follows  that
\[\textrm{either}\quad \mu E+b\subseteq f_{\beta+1}(E)= \frac{E+\beta+1}{\beta}\quad\textrm{or}\quad \mu E+b\subseteq f_0(E)\cup f_1(E).\]
In the first case we have $\beta\mu E+c:=\beta\mu E+b\beta-\beta-1\subseteq E$. So it suffices to   consider the latter case.

Suppose $\mu E+b\subseteq f_0(E)\cup f_1(E)$. Note that $\mu\in(0, 1/\beta)$ and $A=f_{\beta+1}(0)-f_1(\gamma_\beta)$. Then the largest hole of $\mu E+b$ has length $\mu A$, which is strictly less than $A/\beta$. This implies that  either
\[ \mu E+b\subseteq f_{1(\beta+1)}(E)\subseteq f_1(E)=\frac{E+1}{\beta}\]
{or} \begin{equation}\label{eq:kong-21}
\mu E+b\subseteq f_{00}(E)\cup f_{01}(E)\cup f_{10}(E)\cup f_{11}(E)= f_0(\Delta)\cap E.
\end{equation}
 In the first case, we have $\beta\mu E+c\subseteq E$ by letting $c=b\beta-1$.
For the second case as in (\ref{eq:kong-21}) we need some effort.

Since $\mu\in(0,1/\beta)$ and $\beta\ge 3$,   by the first inequality in  (\ref{eq:kong-11}) it follows that
\[\left(\mu \cdot \frac{E+\beta+1}{\beta}+b\right)_{\min}=\mu\left(1+\frac{1}{\beta}\right)+b
>\frac 1\beta\ge \frac{2}{ \beta(\beta-1)}= f_{01}(\ga_\beta).\]
Clearly, by (\ref{eq:kong-21}) we have
\[
\left(\mu \cdot \frac{E+\beta+1}{\beta}+b\right)_{\min}\le\left(\mu E+b\right)_{\max}<f_{\beta+1}(0).
\]
So by Lemma  \ref{belong} (ii) it follows that
\[\mu \cdot \frac{E+\beta+1}{\beta}+b\subseteq f_1(E)=\frac{E+1}{\beta},\]
 or, equivalently
\begin{equation*}
\mu E+T(b):=\mu E+b\beta+\mu\beta+\mu-1\subseteq E.
\end{equation*}
Furthermore,  by the second inequality in (\ref{eq:kong-11}) we have $T(b)>b$.

Repeating the above process we have either $\beta\mu E+c\subseteq E$ for some $c\in{\mathbb R}$ or
\begin{equation}\label{eq:kong-12}
\mu E+T^n(b)\subseteq E\quad{\textrm{for all}}\quad n\in\N.
\end{equation} We will finish the proof by showing that the case in (\ref{eq:kong-12}) is impossible.
It is easy to check that
\begin{equation}\label{t2nb}
T^n(x)=\beta^n\left(x-\frac{1-\mu-\mu\beta}{\beta-1}\right)+\frac{1-\mu-\mu\beta}{\beta-1}\quad\forall n\in\N.
\end{equation}
Therefore, by (\ref{eq:kong-11}) and (\ref{t2nb}) it follows that  $T^n(b)\nearrow +\infty$ as $n\to\infty$. This leads to a contradiction with (\ref{eq:kong-12}).
\end{proof}
\begin{remark}\label{remark}
The proof of Lemma  \ref{mul0} implies the following fact: Let $\mu E+b\subseteq E$ with $0<\mu<1/\beta$. If
\[b\ge \frac{1}{\beta}-\frac{2\mu}{\beta-1}\quad\textrm{and}\quad \mu E+b\subseteq f_0(\Delta)\cap E,\]
 then
$\mu E+b\beta+\mu\beta+\mu-1\subseteq E.$ This fact will be used later.
\end{remark}

\begin{lemma}\label{b*}
Let $\mu E+b\subseteq E$ with $0<\mu<1/\beta$. If
\begin{equation}\label{eq:kong-31}
b\ge \frac{1}{\beta}-\frac{2\mu}{\beta-1}\quad\textrm{and}\quad b=\frac{1-\mu-\mu\beta}{\beta-1},
\end{equation} then   $\mu E+b^*\subseteq E$ for some $b^*\ne b$.
Furthermore, we have $b^*>b$ unless $\mu\le  1/({\beta^2+1})$.
\end{lemma}
\begin{proof}
Suppose $b=({1-\mu-\mu\beta})/({\beta-1})$. Then
\[(\mu E+b)_{\max}=\frac{1}{\beta-1}<\frac{\beta+1}{\beta}=f_{\beta+1}(0).\]
So   $\mu E+b\subseteq f_0(E)\cup f_1(E)$.
By (\ref{eq:kong-31}) it follows that
\[0\le \mu\le \frac{1}{\beta(\beta-1)}.\]
We consider the following two cases.

{\bf Case 1.}  If ${1}/({\beta^2+1})<\mu\le {1}/({\beta(\beta-1)})$, then
\begin{align*}
\left(\mu\cdot\frac{E}{\beta^2}+b\right)_{\max}&=\frac{\mu\ga_\beta}{\beta^2}+\frac{1}{\beta-1}-\mu\ga_\beta=\frac 1{\beta-1}-\mu\frac{(\beta+1)^2}{\beta^2}\\
&< \frac 1{\beta-1}-\frac{1}{\beta^2+1}\cdot\frac{(\beta+1)^2}{\beta^2}\\
&\le\frac 1{\beta}=f_{10}(0),
\end{align*}
where in the last inequality we have used $\beta\ge 3$. So, by
  Lemma  \ref{belong} (i) it follows that
$\mu\beta^{-2}E+b\subseteq  \beta^{-1}E,$
 which is equivalent to
 \[\mu\cdot\frac{E}{\beta}
+\beta b\subseteq E.\]
Since ${1}/({\beta^2+1})<\mu\le {1}/({\beta(\beta-1)})$, it follows  that
\[\left(\mu\cdot\frac{E}{\beta}
+\beta b\right)_{\min}=\frac \beta{\beta-1}-\frac{\beta(\beta+1)}{\beta-1}\mu
\in\left(\frac {2 }{\beta(\beta-1)},\frac{\beta+1}{\beta}\right)=(f_{01}(\ga_\beta), f_{\beta+1}(0)).
\]
 By Lemma  \ref{belong} (ii) it follows that
\[\mu\cdot\frac{E}{\beta}
+\beta b\subseteq\frac{E+1}\beta.\]
 Thus
$\mu E+b^*\subseteq E$ with
$b^*:=\beta^2 b-1.$
Since  $\beta\ge 3$ and $\mu \le {1}/{(\beta(\beta-1))}$, by an easy computation we obtain  $b^*>b$.

{\bf Case 2.} If $\mu\le {1}/({\beta^2+1})$, then by using $\beta\ge 3$ and (\ref{eq:kong-31}) we have
\[\left(\mu\cdot\frac{E+1}{\beta}
+b\right)_{\min}=\frac{\mu}{\beta}+\frac{1-\mu-\mu\beta}{\beta-1}\ge \frac 1{\beta}\ge\frac{2}{\beta(\beta-1)}=f_{01}(\ga_\beta).\] Therefore, by Lemma  \ref{belong} (ii) we get
\[\mu\cdot\frac{E+1}{\beta}
+b\subseteq \frac{E+1}\beta.\] This is equivalent to $\mu E+b^*\subseteq E$ with
$b^*=\mu-1+\beta b.$ Obviously, by (\ref{eq:kong-31}) we have $b^*< b$.
\end{proof}

Now we are ready to prove the following lemma.
\begin{lemma}\label{mu2}
Let $\mu E+b\subseteq E$ with $0<|\mu|<1/\beta$. Then there exists $c\in {\mathbb R}$ such that
\[\beta\mu E+c\subseteq E.\]
\end{lemma}
\begin{proof}
Since the case for $\mu<0$ can be proved similarly, we only consider the case for $\mu>0$.

By repeating the same process as in the proof of Lemma  \ref{mul0} we have either $\beta\mu E+c\subseteq E$ for some $c\in{\mathbb R}$, or
\begin{equation}\label{eq:kong-22}
\mu E+b\subseteq   f_0(\Delta)\cap E.
\end{equation}
In the following it suffices to consider the case in (\ref{eq:kong-22}). Clearly, by (\ref{eq:kong-22}) it follows that
\[
0\le b=(\mu E+b)_{\min}\le f_0\left(\ga_\beta\right)=\frac{\ga_\beta}{\beta}.
\]
If $b\ge f_{01}(\ga_\beta)=2/(\beta(\beta-1))$, then by Lemma  \ref{belong} (ii) it follows that
$\mu E+b\subseteq f_1(E)=\beta^{-1}({E+1}),$
which implies
$\mu\beta E+c\subseteq E$ with $c=b\beta-1$. So we only need to consider $b\in\left[0,2/(\beta(\beta-1))\right)$. It is convenient to divide the proof into the following two cases.

{\bf Case 1.} $0\le b<  1/{\beta}- {2\mu}/{(\beta-1)}$.
Then
\[\left(\mu\cdot\frac{E+1}{\beta}+b\right)_{\max}=\mu\frac{\ga_\beta+1}{\beta}+b< \frac{2\mu}{\beta-1}+\frac 1\beta-\frac{2\mu}{\beta-1}=\frac 1\beta=f_{10}(0).
\]
 Therefore, we have
 \[\mu\cdot\frac{E+1}{\beta}+b\subseteq \frac E{\beta}\]
  by Lemma  \ref{belong} (i), which is equivalent to $\mu E+\beta b+\mu\subseteq E$.
Define $S(x):=\beta x+\mu$. Then
\begin{equation}\label{sn}
S^n(x)=\beta^n\left(x+\frac{\mu}{\beta-1}\right)-\frac{\mu}{\beta-1}.
 \end{equation}
 Note that $S(b)\ne b$. So there exists a unique positive integer $m_0$ such that
 \[S^{m_0}(b)\ge\frac 1{\beta}-\frac{2\mu}{\beta-1}\quad \textrm{and}\quad
S^{m_0-1}(b)< \frac 1{\beta}-\frac{2\mu}{\beta-1}.\]

We continue the above process by replacing $\mu E+b\subseteq E$ with $\mu E+S(b)\subseteq E$ to draw
\[
\mu E+S^{m_0}(b)\subseteq E,\]
 {which} is  reduced to the below case.

\medskip

{\bf Case 2.} $1/{\beta}-{2\mu}/{(\beta-1)}\le b<2/(\beta(\beta-1))$.
There are three cases to consider.

{
{\bf (2A).} $b>(1-\mu-\mu\beta)/(\beta-1)$. Then by Lemma \ref{mul0} there exists $c\in\R$ such that $\beta\mu E+c\subseteq E$.

{\bf(2B).} $b<(1-\mu-\mu\beta)/(\beta-1)$. Then by Remark \ref{remark} we have
\[\mu E+T(b)=\mu E+b\beta+\mu\beta+\mu-1\subseteq E.\]
 Since $T(b)<b$ and
 \[T^n(b)=\beta^n\left(b-\frac{1-\mu-\mu\beta}{\beta-1}\right)+\frac{1-\mu-\mu\beta}{\beta-1},\]
  so there exists a unique positive integer $n_0$ such that
\[T^{n_0-1}(b)\ge\frac 1{\beta}-\frac{2\mu}{\beta-1}\quad \textrm{and}\quad T^{n_0}(b)<\frac 1{\beta}-\frac{2\mu}{\beta-1}.\]
Then 
we have $\mu E+T^{n_0}(b)\subseteq E$ by using Remark \ref{remark} for $n_0-1$ times.

Suppose $\mu E+T^{n_0}(b)\subseteq E$. If $0\le b':=T^{n_0}(b)< 1/\beta-\mu \ga_\beta$, then $(\mu E+b')_{\max}<1/\beta$. By Lemma \ref{belong} (i) we have $\mu E+b'\subseteq f_0(E)$, which implies $\beta \mu E+c\subseteq E$ with $c=\beta b'$.

If $1/\beta-\mu\ga_\beta\le b'<1/\beta-2\mu/(\beta-1)$, then by Case 1 it follows that
$
\mu E+\beta b'+\mu \subseteq E.$
We claim
\[
f_{01}(\ga_\beta)<\left(\mu\frac{E+\beta+1}{\beta}+\beta b'+\mu\right)_{\min}<f_{\beta+1}(0),
\]
or equivalently,
\begin{equation}\label{eq:kong-51}
\frac{2}{\beta(\beta-1)}<2\mu +\beta b'+\frac{\mu}{\beta}<\frac{\beta+1}{\beta}.
\end{equation}
Since $b'<1/\beta-2\mu/(\beta-1)$, one can verify the right inequality of (\ref{eq:kong-51}) directly. For the left inequality in (\ref{eq:kong-51}) we note from the hypothesis in Case (2B) that
\[
\frac{1}{\beta}-\frac{2\mu}{\beta-1}\le b<\frac{1-\mu-\mu\beta}{\beta-1}.
\]
This implies $\mu<1/(\beta(\beta-1))$. Using this and the fact that $\beta\ge 3$ we can prove the left inequality of (\ref{eq:kong-51}). So by (\ref{eq:kong-51}) and Lemma \ref{belong} (ii) it follows that
\[
\mu\frac{E+\beta+1}{\beta}+\beta b'+\mu\subseteq \frac{E+1}{\beta},
\]
which implies
\[
\mu E+\beta^2 b'+2\beta\mu+\mu-1\subseteq E.
\]
Using $\mu<1/(\beta(\beta-1))$ and the fact $\beta\ge 3$ one can verify that
\[
b'':=\beta^2 b'+2\beta\mu+\mu-1>\frac{1-\mu-\mu\beta}{\beta-1}>\frac{1}{\beta}-\frac{2\mu}{\beta-1}.
\]
Then by Lemma \ref{mul0} we conclude that $\beta\mu E+c\subseteq E$ for some $c\in\R$.

{\bf(2C).} $b=(1-\mu-\mu\beta)/(\beta-1)$. Then by Lemma  \ref{b*} we have $\mu E+b^*\subseteq E$ with $b^*\ne b$.
Furthermore, by Lemma \ref{b*} it follows that  if $\mu>1/(\beta^2+1)$, then we have $b^*>b$. In this case we can conclude by Lemma \ref{mul0} that $\beta\mu E+c\subseteq E$ for some $c\in\R$.

In the following we assume $\mu\le 1/(\beta^2+1)$. Then by the proof of Lemma \ref{b*} it follows that $\mu E+b^*\subseteq E$ with
\begin{equation}\label{eq:kong-61}
b^*=\mu-1+\beta b=\mu-1+\beta\frac{1-\mu-\mu\beta}{\beta-1}<b.
\end{equation}
So, either we can get $\beta\mu E+c\subseteq E$ for some $c\in\R$, or we can reduce to Case 1 that
\[
b^*<\frac{1}{\beta}-\frac{2\mu}{\beta-1}\quad\textrm{and}\quad S^n(b^*)=b\quad\textrm{for some }n\in\N,
\]
where $S(x)=\beta x+\mu$.

Suppose $S^n(b^*)=b=(1-\mu-\mu\beta)/(\beta-1)$. Then by (\ref{sn}) it follows that
\[
b^*=\frac{1-\mu\beta-\mu\beta^n}{\beta^n(\beta-1)}.
\]
Combined with (\ref{eq:kong-61}) we obtain
\[
\mu=\frac{\beta^n-1}{\beta^{n+2}-\beta}.
\]
Using   $\mu\le 1/(\beta^2+1)$ and $\beta\ge 3$ this implies $n=1$. So
\[
\mu=\frac{1}{\beta(\beta+1)}\quad\textrm{and} \quad b^*=\frac{1-2\mu\beta}{\beta(\beta-1)}=\frac{1}{\beta(\beta+1)}.
\]
We will finish the proof in Case (2C) by proving
\begin{equation}\label{eq:kong-62}
\mu E+b^*=\frac{E+1}{\beta(\beta+1)}\nsubseteq E.
\end{equation}

Suppose $(E+1)/(\beta(\beta+1))\subseteq E$. Since
\[
\left(\frac{E+1}{\beta(\beta+1)}\right)_{\max}=\frac{\ga_\beta+1}{\beta(\beta+1)}=\frac{2}{\beta^2-1}<\frac{2}{\beta(\beta-1)}=f_{01}(\gamma_\beta),
\]
by Lemma \ref{belong} (i) it follows that
\[
\frac{E+1}{\beta(\beta+1)}\subseteq\frac{E}{\beta}\quad\Longrightarrow\quad \frac{E+1}{\beta+1}\subseteq E.
\]
One can check that $\hat \mu E+\hat b:=(E+1)/(\beta+1)\subseteq E$ satisfies the conditions in Lemma \ref{mul0}. Then by Lemma \ref{mul0} we get
\[
\frac{\beta}{\beta+1}E+c\subseteq E\quad\textrm{for some }c\in\R.
\]
This leads to a contradiction with Lemma \ref{mu1} since $\beta/(\beta+1)>1/\beta$.
}

Therefore, for $\mu E+b\subseteq E$ with $\mu\in(0, 1/\beta)$ we must have $\beta\mu E+c\subseteq E$ for some $c\in \R$. This completes the proof.
\end{proof}

\begin{proof}[Proof of Proposition \ref{prop:contration-ratio}]
By Lemma  \ref{mu1}, we have either $0<|\mu|<\frac 1{\beta}$ or $\mu=\frac 1{\beta}$. Suppose $\mu\ne \beta^{-n}$ for any positive integer $n$. Then there exists a positive integer $k$
such that ${\beta^{-(k+1)}}<\mu<{\beta^{-k}}$ or ${\beta^{-(k+1)}}\le -\mu<{\beta^{-k}}$. Using Lemma  \ref{mu2} for $k$ times yields
\begin{equation*}
\beta^k\mu E+c_k\subseteq E\,\,{\textrm{for some}}\,c_k\in{\mathbb R},
\end{equation*}
where  $\beta^k\mu\in (1/\beta,1)$ or $\beta^k\mu\in(-1, -1/\beta]$. This  leads to a contradiction with Lemma  \ref{mu1}.
\end{proof}

\subsection{Determination of $b$ for $\mu E+b\subset E$}
First consider the case for $\mu=\beta^{-1}$.
\begin{lemma}\label{b1}
If $g(E)=\beta^{-1}E+b\subseteq E$, then $b\in\{f_0(0),f_1(0),f_{\beta+1}(0)\}$
\end{lemma}
\begin{proof}
Suppose $\beta^{-1} E+b\subseteq E$. Then by Lemma \ref{sep}  we have
\[\textrm{either}\quad \frac E\beta+b\subseteq f_{\beta+1}(E)= \frac{E+\beta+1}{\beta}\quad \textrm{or}\quad \frac E\beta+b\subseteq f_0(E)\cup f_1(E).\]
 In the first case, we clearly have $b=f_{\beta+1}(0)$, and we pay attention to the second case.

Suppose $ \beta^{-1}E+b\subseteq f_0(E)\cup f_1(E)$.  Then either $ \beta^{-1}E+b\subseteq f_0(\Delta)\cap E$ or $\mu E+b$ contains the hole  $H_1=\left(f_{11}\left(\ga_\beta\right),f_{1(\beta+1)}(0)\right)$  of length $A/\beta$, where $A$ is the length of the largest hole
$H=\left(f_1\left(\ga_\beta\right),f_{\beta+1}(0)\right)$ of $E$.
If $\beta^{-1}E+b\subseteq f_0(\Delta)\cap E$, then $b=0=f_1(0)$.

If $\beta^{-1}E+b$ contains the hole $H_1$, note that every hole of $\beta^{-1}E+b$ is mapped by a hole of $E$ with scaling $1/\beta$, then this hole must be mapped by exactly the largest hole $H$ of $E$. Thus,
\[\frac 1\beta\cdot f_{1}\left(\ga_\beta\right)+b=f_{11}\left(\ga_\beta\right)\quad\textrm{and}\quad \frac1\beta\cdot f_{\beta+1}(0)+b=f_{1(\beta+1)}(0).\]
This yields  $b=1/\beta=f_1(0)$, completing the proof.
\end{proof}

\begin{proof}[Proof of Theorem \ref{ifs}]
Suppose $g(E)=\mu E+b\subseteq E$ with $0<|\mu|<1$ and $b\in\R$. By Proposition \ref{prop:contration-ratio} it follows that $\mu=\beta^{-n}$ for some $n\in\N$. In the following we will prove by induction on $n$ that if $\mu=\beta^{-n}$, then $b=f_{\bf i}(0)$ for some ${\bf i}\in\Omega^n$.

When $n=1$, this has been proved in Lemma \ref{b1}.
Assume $n\ge 1$ and suppose $b\in \{f_{\bf i}(0):{\bf i}\in\Omega^n\}$ for $\mu=\beta^{-n}$. We will prove that $b\in\{f_{\bf i}(0):{\bf i}\in\Omega^{n+1}\}$ for
$\mu=\beta^{-(n+1)}$.

{\bf Case A.} $b\in f_{\beta+1}(E)$. {Then} by Lemma \ref{belong} (iii) we have $\beta^{-(n+1)}E+b\subseteq f_{\beta+1}(E)$. Thus
\[\beta^{-n}E+\beta b-{\beta}-1\subseteq E.\] By induction we have
$\beta b-\beta-1=f_{\bf j}(0)$ for some ${\bf j}\in\Omega^n$. It follows that $b=f_{(\beta+1)\bf j}(0)$ with $(\beta+1){\bf j}\in\Omega^{n+1}$.

{\bf Case B.} $b\in f_0(E)\cup f_1(E)$. If $b\ge  \beta^{-1}$, then by Lemma   \ref{belong} (ii) we have $\beta^{-(n+1)}E+b\subseteq f_1(E)$. In this case,
$b=f_{1\bf j}(0)$ for some ${\bf j}\in\Omega^n$. Now we suppose $b< \beta^{-1}$.
If $b< \beta^{-1}-\beta^{-(n+1)}\ga_\beta$, then $\left(\beta^{-(n+1)}E+b\right)_{\max}<1/{\beta}$. By Lemma   \ref{belong} (i) we have
\[\beta^{-(n+1)}E+b\subseteq \frac E\beta,\]
and therefore
$b=f_{0\bf j}(0)$ for some ${\bf j}\in\Omega^n$.
In the following we assume  $\beta^{-1}-{\beta^{-(n+1)}}\ga_\beta\le b<\beta^{-1}$. Then there exists a unique $m\ge n+1\ge 2$ such that
\begin{equation}\label{m}
\frac{1}\beta- \frac{\ga_\beta}{\beta^{m+1}}>b\ge \frac{1}\beta- \frac{\ga_\beta}{\beta^{m}}.
\end{equation}
We claim that
\begin{equation}\label{eq:kong-41}
b\in\left\{f_{01(\beta+1)^{m-2}0}(0),f_{01(\beta+1)^{m-2}1}(0),f_{01(\beta+1)^{m-1}}(0)\right\}.
\end{equation}

By (\ref{m}) we have $\left(\beta^{-(m+1)}E+b\right)_{\max}< 1/\beta$. So  by Lemma   \ref{belong} (i) it follows that
$\beta^{-(m+1)}E+b\subseteq\beta^{-1} E,$ which is equivalent to
\[\beta^{-m} E+b\beta\subseteq E.\]
 Then  by (\ref{m}) and using $\beta\ge 3$ it follows that
\begin{eqnarray*}
b\beta\ge1-\frac{\ga_\beta}{\beta^{m-1}}\ge 1-\frac{ \ga_\beta}{\beta}\ge \frac1\beta\quad{\textrm{and}}\quad
\beta^{-m}\ga_\beta+b\beta<1<\frac{\beta+1}{\beta}.
\end{eqnarray*}
So, by Lemma   \ref{belong} (ii) we have  $\beta^{-m} E+b\beta\subseteq f_1(E)$, which implies
\begin{equation*}
\beta^{-(m-1)}E+b\beta^2-1\subseteq E.
\end{equation*}
If $m=2$, then by Lemma \ref{b1} we prove (\ref{eq:kong-41}). Otherwise,  for $m\ge 3$ it follows from   (\ref{m})   that
\begin{equation}\label{>}
b\beta^2-1\ge \beta-1-\frac {\ga_\beta}{\beta^{m-2}}\ge \ga_\beta-\frac {\ga_\beta}{\beta^{m-2}}\ge \frac{\beta+1}{\beta}.
\end{equation}
Therefore, by Lemma \ref{belong} (iii)
we have
$
\beta^{-(m-1)}E+b\beta^2-1\subseteq \beta^{-1}({E+\beta+1})
$
which implies
\[\beta^{-(m-2)}E+\beta(b\beta^2-1)
-(\beta+1)\subseteq E.
\]
If $m=3$, then by Lemma (\ref{b1}) we get (\ref{eq:kong-41}). Otherwise, for $m\ge 4$ it follows
 by (\ref{>}) that  \[\beta(b\beta^2-1)
-(\beta+1)\ge \beta\left(\ga_\beta-\frac {\ga_\beta}{\beta^{m-2}}\right)-(\beta+1)=\ga_\beta-\frac {\ga_\beta}{\beta^{m-3}}\ge \frac{\beta+1}{\beta}.\]

We continue the above process for $m-3$ times to get
\begin{equation*}
\frac E\beta+\beta^{m-2}(b\beta^2-1)-(1+\beta)(1+\beta+\cdots+\beta^{m-3})\subseteq E.
\end{equation*}
Therefore by Lemma   \ref{b1}, we have
\[\beta^{m-2}(b\beta^2-1)-(1+\beta)(1+\beta+\cdots+\beta^{m-3})\in \left\{f_0(0),f_1(0),f_{\beta+1}(0)\right\}.\]
This proves (\ref{eq:kong-41}), establishing the claim.

{\bf(B1).} $b=f_{01(\beta+1)^{m-2}0}(0)=f_{01(\beta+1)^{m-2}}(0)$. 
If $m=n+1$, then $b=f_{01(\beta+1)^{n-1}}(0)$, and we are done. Now we suppose $m\ge n+2$. Then we claim that $\beta^{-(n+1)}E+b\subseteq E$ is impossible.

Note that $y=\pi(0^{m-n-2}1 0^k (\beta+1)^\f)\in E$ for any  $k\ge 2$, where $\pi$ is the natural projection map from the symbolic space $\Om^\N$ to $E$.
Then one can check for $k$ sufficiently large that
\[
\beta^{-(n+1)}y+f_{01(\beta+1)^{m-2}}(0)\in H_{01(\beta+1)^{m-2}}=\left(f_{01(\beta+1)^{m-2}1}\left(\ga_\beta\right), f_{01(\beta+1)^{m-1}}(0)\right),
\]
where the open interval $H_{01(\beta+1)^{m-2}}$ is a hole of $E$.
So $\beta^{-(n+1)}E+b\nsubseteq E$.

{\bf(B2).} $b=f_{01(\beta+1)^{m-2}1}(0)$.
Here we also show that  $\beta^{-(n+1)}y+b\nsubseteq E$ for all $m\ge n+1$. Note that $z=\pi(0^{m-n-1}(\beta+1)0^k(\beta+1)^\f)\in E$ for all  $k\in\mathbb N$.  Observe by using $\beta\ge3$  that
\begin{equation}\label{eq:kong}
f_{(\beta+1)1}\left(\ga_\beta\right)\le f_1(0)+f_{\beta+1}(0).
\end{equation}
 Then by (\ref{eq:kong}) one can prove for $k$ sufficiently large that
\[
\beta^{-(n+1)}z+f_{01(\beta+1)^{m-2}1}(0)\in H_{01(\beta+1)^{m-1}}=\left(f_{01(\beta+1)^{m-1}1}\left(\ga_\beta\right), f_{01(\beta+1)^m}(0)\right)
\]
with $H_{01(\beta+1)^{m-1}}$ being a hole of $E$. So $\beta^{-(n+1)}y+b\nsubseteq E$.

{\bf(B3).} $b=f_{01(\beta+1)^{m-1}}(0)$. By the same argument as in  (B1) it follows that $\beta^{-(n+1)}E+b\nsubseteq E$ for all $m\ge n+1$.

Hence, for $\beta^{-(n+1)}E+b\subseteq E$ we have $b\in\set{f_{\bf j}(0): {\bf j}\in\Om^{n+1}}$. This completes the proof by induction.
\end{proof}

\section{Spectrum and unique/multiple expansions}

In this section we first prove Proposition \ref{prop:spectrum}, which shows that the spectrum of $E$ is constant if $\beta\ge 3$.
\begin{proof}[Proof of Proposition \ref{prop:spectrum}]First, let $n=0$ and $d_0=1$, then we have $\ell_\beta\le 1$. Now we prove the other direction. It suffices to prove the following claim:  \emph{for any $n\ge 0$ and any $\sum_{i=0}^n d_i\beta^i\ne 0$ with $d_n\ne 0$ we have $|\sum_{i=0}^n d_i\beta^i|\ge 1$.}  We will prove this by induction on $n$.

Clearly, the claim holds true for $n=0$. Suppose it holds for all $n<k$ for some positive integer $k$. Now we consider  $n=k$. Let $\sum_{i=0}^{k}d_i\beta^i\ne 0$ with $d_i\in\set{0, \pm 1, \pm\beta, \pm(\beta+1)}$ and $d_{k}\ne 0$. Then $|d_{k}|\in\set{1,\beta, \beta+1}$. We consider the following three cases.

Case I. $|d_k|=\beta$ or $\beta+1$. Then by using $|d_i|\le \beta+1$ and $\beta\ge 3$ it follows that
\[
\left|\sum_{i=0}^k d_i\beta^i\right|\ge \beta\cdot\beta^{k}-(\beta+1)\sum_{i=0}^{k-1}\beta^i=\beta^{k+1}\left(1-\frac{\beta+1}{\beta(\beta-1)}\right)+\frac{\beta+1}{\beta-1}>1
\]
as desired.

Case II. $|d_k|=1$. Without loss of generality we assume $d_k=1$. If $d_{k-1}=-\beta$ or $-(\beta+1)$, then we can rewrite
\[
\sum_{i=0}^k d_i\beta^i=\sum_{i=0}^{k-1}d_i'\beta^i
\]
with $d_{k-1}'=d_{k-1}+\beta$ and $d_i'=d_i$ for all $0\le i<k-1$. By the induction hypothesis it follows that $|\sum_{i=0}^k d_i\beta^i|\ge 1$.
If $d_{k-1}\notin\set{-\beta, -(\beta+1)}$, then by using $|d_i|\le \beta+1$ and $\beta\ge 3$ it follows that
\begin{align*}
\left|\sum_{i=0}^k d_i\beta^i\right|&\ge \beta^{k}-\beta^{k-1}-(\beta+1)\sum_{i=0}^{k-2}\beta^i
=\beta^{k-1}\frac{\beta^2-3\beta}{\beta-1}+\frac{\beta+1}{\beta-1}>1.\end{align*}

By induction this proves the claim, and then completes the proof.
\end{proof}

In the following we will investigate the Hausdorff dimensions and Hausdorff measures of $E$ and $E^{(k)}$. Note that $E^{(k)}$ is the set of points in
$E$ having precisely $k$ different codings {with respect to the} alphabet $\Omega$.
For this we need  the following property of $E$.
\begin{lemma}\label{lem:overlap}
$f_0(E)\cap f_1(E)=f_{0(\beta+1)}(E)=f_{11}(E)$.
\end{lemma}
\begin{proof}
By $f_{0(\beta+1)}=f_{11}$ we have
\begin{align*}
f_0(E)\cap f_1(E)\subseteq f_0(E)\cap f_1(\Delta)&=f_{0(\beta+1)}(E)\cap f_1(\Delta)\\
&=f_{11}(E)\cap f_1(\Delta)=f_{11}(E)=f_{0(\beta+1)}(E).
\end{align*}
On the other hand, since $f_{0(\beta+1)}(E)\subseteq f_0(E)$ and $f_{11}(E)\subseteq f_1(E)$, using $f_{0(\beta+1)}=f_{11}$ again it follows that
\[
f_0(E)\cap f_1(E)\supseteq f_{0(\beta+1)}(E)\cap f_{11}(E)=f_{11}(E)=f_{0(\beta+1)}(E).
\]
This completes the proof.
\end{proof}

Now we are ready to prove Proposition \ref{prop:dim-E-U}.

\begin{proof}[Proof of Proposition \ref{prop:dim-E-U}]
The proof is similar to that of \cite[Theorem 2]{DJKLX}. For completeness we sketch the main idea.

Note by Lemma \ref{lem:overlap} that for any point $x\in E$, if $x$ has a coding containing the block $11$ or $0(\beta+1)$, then $x$ has at least two different codings by observing the substitution $11\sim 0(\beta+1)$. On the other hand, if $x$ has two different codings, say $(c_i)$ and $(d_i)$ (without loss of generality we assume $c_1< d_1$), then we must have $c_1=0$ and $d_1=1$. So, $x\in f_0(E)\cap f_1(E)$. By Lemma \ref{lem:overlap} it follows that $x\in f_{0(\beta+1)}(E)=f_{11}(E)$. This implies that $c_1c_2=0(1+\beta)$ and $d_1d_2=11$. Therefore, $x\in E$ has multiple codings if and only if its codings contain the block $11$ or $0(\beta+1)$.

Note that $f_{11}=f_{0(\beta+1)}$. One can verify that the set $E$ is a graph-directed set satisfying the open set condition. More precisely, let $X_A$ be the subshift of finite type with the forbidden block $0(\beta+1)$. Then
\[
X_A=\set{(d_i)\in\Om^\N: A_{d_i, d_{i+1}}=1},
\]
where $A$ is the transition matrix with states $0, 1$ and $\beta+1$ given by
\[
A=\left(
\begin{array}{ccc}
1&1&0\\
1&1&1\\
1&1&1
\end{array}\right).
\]
Then $E=\pi_\beta(X_A)$ is a graph-directed set.  By a well-known result of \cite{MW} it follows that
\[
\dim_H E=\frac{\log r_A}{\log \beta}=\frac{\log(3+\sqrt{5})-\log 2}{\log \beta},
\]
where $r_A={({3+\sqrt{5}})/{2}}$ is the spectral radius of $A$. Since the matrix $A$ is {irreducible}, or equivalently, the subshift of finite type $X_A$ is transitive with respect to the left shift, we conclude that $\mathcal H^{\dim_H E}(E)\in(0, \f)$. Then (i) follows from the second part that
\[
\dim_H E^{(k)}<\dim_H E\quad\textrm{for any }k\ne 2^{\aleph_0}\quad\textrm{and}\quad E=E^{(2^{\aleph_0})}\cup E^{(\aleph_0)}\cup\bigcup_{k=1}^\f E^{(k)},\]
which we will prove below.

Now we consider the subset $E^{(k)}$. Let $U:=E^{(1)}$ be the set of $x\in E$ having a unique coding. Observe that a point in $E$ has multiple codings if and only if its codings contain the block $11$ or $0(\beta+1)$. In other words, the set $U$ consists of all  $x\in E$ with its unique coding   containing neither $11$ nor $0(\beta+1)$. So, $U$ is also a graph-directed set satisfying the open set condition. Let $X_B$ be the subshift of finite
type with forbidden blocks $11$ and $0(\beta+1)$. Then
\[
X_B=\set{(d_i)\in\Om^\N: B_{d_i, d_{i+1}}=1},
\]
where $B$ is the transition matrix on the alphabet $\Om=\set{0, 1,\beta+1}$ defined by
\[
B=\left(
\begin{array}{ccc}
1&1&0\\
1&0&1\\
1&1&1
\end{array}\right).
\]
Therefore, $U=\pi_\beta(X_B)$. So, the Hausdorff dimension of $U$ is given by (cf.~\cite{MW})
\[
\dim_H U=\frac{\log r_B}{\log \beta},
\]
where $r_B\approx 2.24698$ is the spectral radius of $B$, which satisfies the equation $r^3-2r^2-r+1=0$. Furthermore, since the matrix $B$ is also irreducible,   the corresponding Hausdorff measure $\mathcal H^{\dim_H U}(U)\in(0, \f)$.

Let $\mathbf U$ be the set of all unique codings of {points in $U$}. Note that for  any $k\in\N$,  any $x\in E^{(k)}$ has precisely $k$ different codings, and all of these codings must be end in $\mathbf U$. Thus,
\[
E^{(k)}\subseteq\bigcup_{n=0}^\f \bigcup_{\mathbf i\in\Om^n} f_{\mathbf i}(U),
\]
which implies $\dim_H E^{(k)}\le \dim_H U$.

 If $k=2^m$ for some $m\in\N$, then one can verify that
\begin{equation}\label{eq:mar-18-1}
 \Lambda_{n,m}:=\pi_\beta(\set{d_1\ldots d_n(0(\beta+1))^m c_1c_2\ldots: d_n=c_1=\beta+1,~d_1\ldots d_n\in B_n(\mathbf U), (c_i)\in\mathbf U})
\end{equation}
is a subset of $E^{(2^m)}$ for any $n\in\N$, where $B_n(\mathbf U)$ is the set of admissible blocks of length $n$ in $\mathbf U$. This implies that
\[
\dim_H E^{(k)}\ge \dim_H\Lambda_{n,m}=\dim_H U.
\]
Furthermore, note that for different $n, n'$ the sets $\Lambda_{n,m}$ and $\Lambda_{n',m}$ are disjoint. Write $t:=\dim_H U$ and denote by $U(\beta+1)$   the set of $x\in U$ with its unique coding beginning with $\beta+1$. Since $X_B$ is transitive, we have
\begin{equation}\label{eq:mar-18-2}
\mathcal H^t(U(\beta+1))>0.
\end{equation} So, by  (\ref{eq:mar-18-1}) it follows that
\begin{align*}
\mathcal H^{t}(E^{(2^m)})&\ge \mathcal H^{t}(\bigcup_{n=1}^\f\Lambda_{n,m})=\sum_{n=1}^\f\mathcal H^{t}(\Lambda_{n,m})\\
&=\sum_{n=1}^\f\sum_{d_1\ldots d_n\in B_n(\mathbf U), d_n=\beta+1} \mathcal H^t(f_{d_1\ldots d_n(0(\beta+1))^m}(U(\beta+1)))\\
&=\beta^{-2mt}\mathcal H^t(U(\beta+1))\sum_{n=1}^\f \left(\sum_{d_1\ldots d_n\in B_n(\mathbf U), d_n=\beta+1}\beta^{-nt}\right)\\
&\ge \beta^{-2mt}\mathcal H^t(U(\beta+1))\sum_{n=1}^\f C=\f,
\end{align*}
where the last inequality follows by using the Perron-Frobenius Theorem (cf.~\cite{LM})  that
\[
 \sum_{d_1\ldots d_n\in B_n(\mathbf U), d_n=\beta+1}\beta^{-nt} \ge C>0
\]
for any $n\ge 1$.

Finally, we prove that $E^{(k)}=\emptyset$ for any other $k\ne 2^m$. Let  $x\in E^{(k)}$. Then $x$ has multiple codings. So there exists a smallest integer $n_1\ge 0$ such that
\[T^{n_1}(x)\in f_0(E)\cap f_1(E)=f_{0(\beta+1)}(E)=f_{11}(E),\]
where $T: E\to E$ is the inverse map of $f_{0}, f_1$ and $f_{\beta+1}$. This implies that  all codings of $x_1:=T^{n_1}(x)$ begin with either $0(\beta+1)$ or $11$.  Note that there exists a unique word $d_1\ldots d_{n_1}\in B_{n_1}(\mathbf U)$ such that $x=f_{d_1\ldots d_{n_1}}(x_1)$. So, any coding of $x$ either begins with $d_1\ldots d_{n_1}0(\beta+1)$ or begins with $d_1\ldots d_{n_1} 11$. {Since $f_{d_1\ldots d_{n_1}0(\beta+1)}=f_{d_1\ldots d_{n_1}11}$,}  there exists a unique $y_1\in E$ such that
\[
x=f_{d_1\ldots d_{n_1}0(\beta+1)}(y_1)=f_{d_1\ldots d_{n_1}11}(y_1).
\]
Now we proceed with the same argument on $y_1$ instead of $x$. Then there exist a smallest integer $n_2\ge 0$,  a unique word $d_{n_1+1}\ldots d_{n_1+n_2}$ and a unique $y_2\in E$ such that
\[
y_1=f_{d_{n_1+1}\ldots d_{n_1+n_2}0(\beta+1)}(y_2)=f_{d_{n_1+1}\ldots d_{n_1+n_2}11}(y_2).
\]

If there exists $m\in\N$ such that the above procedure stops after $m$ steps, then $x$ has precisely $2^m$ different codings. If the above procedure never stops, then $x$ has a continuum of different codings. So,  $E^{(k)}=\emptyset$ for all $k\ne 2^m$.  This completes the proof.
\end{proof}

\section*{Acknowledgements}
{The first author was supported by  NSFC No.~11971079 and   the Fundamental
and Frontier Research Project of Chongqing No.~cstc2019jcyj-msxmX0338.}

\end{document}